\documentclass[a4paper,10pt,reqno]{amsart}
\usepackage{fullpage}
\usepackage[utf8]{inputenc}
\usepackage{hyperref}
\usepackage{amsmath}
\usepackage{amsfonts}
\usepackage{amssymb}
\usepackage{amsthm}
\usepackage{latexsym}
\usepackage{xcolor}
\usepackage{tikz}
\usepackage{mathtools}

\providecommand\given{}
\newcommand\SetSymbol[1][]{%
	\nonscript\:#1{:}
	\allowbreak
	\nonscript\:
	\mathopen{}}
\DeclarePairedDelimiterX\Set[1]\{\}{%
	\renewcommand\given{\SetSymbol}
	#1
}

\DeclarePairedDelimiter{\abs}{\lvert}{\rvert}
\DeclarePairedDelimiterXPP{\norm}[2]{}{\lVert}{\rVert}{_{#2}}{#1}
\DeclarePairedDelimiterXPP{\Altnorm}[3]{}{\lVert}{\rVert}{_{#2}^{#3}}{#1}
\DeclarePairedDelimiterXPP{\hNorm}[2]{}{\lvert}{\rvert}{_{#2}}{#1}
\DeclarePairedDelimiterXPP{\HNorm}[2]{}{\lvert}{\rvert}{_{#2}^\circ}{#1}

\definecolor{darkgreen}{rgb}{0.0, 0.2, 0.13}
\definecolor{darkolivegreen}{rgb}{0.33, 0.42, 0.18}
\definecolor{chamoisee}{rgb}{0.63, 0.47, 0.35}
\definecolor{cerulean}{rgb}{0.0, 0.48, 0.65}
\definecolor{coolgrey}{rgb}{0.55, 0.57, 0.67}

\DeclareMathOperator{\supp}{supp}

\DeclareMathOperator{\Id}{Id}

\newcommand{\N}{\mathbb{N}}
\newcommand{\R}{\mathbb{R}}
\newcommand{\RR}{\mathcal{R}}
\newcommand{\C}{\mathbb{C}}

\newcommand{\E}{\mathcal{E}}
\newcommand{\D}{\mathcal{D}}

\newcommand{\An}{\mathcal{A}}
\newcommand{\G}{\mathcal{G}}

\newcommand{\bG}{\mathbf{G}}

\newcommand{\bL}{\mathbf{L}}
\newcommand{\bN}{\mathbf{N}}
\newcommand{\bV}{\mathbf{V}}
\newcommand{\bT}{\mathbf{T}}

\newcommand{\alp}{{\lvert\alpha\rvert}}

\newcommand{\xit}{{\lvert\xi\rvert}}

\newcommand{\nut}{\lvert\nu\rvert}
\newcommand{\nupt}{\lvert\nu^\prime\rvert}

\newcommand{\Beu}[2]{\mathcal{E}^{( #1 )} ( #2 )}
\newcommand{\Rou}[2]{\mathcal{E}^{\{ #1 \}} ( #2 )}
\newcommand{\DC}[2]{\mathcal{E}^{[ #1 ]} ( #2 )}

\newcommand{\fhRou}[2]{\widetilde{\mathcal{O}}^{\{#1\}}(#2)}

\newcommand{\vBeu}[3][P]{\mathcal{E}^{ (#2) } ( #3 ; #1 )}
\newcommand{\vRou}[3][P]{\mathcal{E}^{\{ #2 \}} \left( #3; #1 \right)}
\newcommand{\vDC}[3][P]{\mathcal{E}^{[ #2 ]} ( #3 ; #1  )}

\newcommand{\bM}{\mathbf{M}}

\newcommand{\eps}{\varepsilon}

\newcommand{\hol}{\mathcal{O}}

\theoremstyle{plain}
\newtheorem{Thm}{Theorem}[section]
\newtheorem{Prop}[Thm]{Proposition}
\newtheorem{Lem}[Thm]{Lemma}
\newtheorem{Cor}[Thm]{Corollary}

\theoremstyle{definition}
\newtheorem{Def}[Thm]{Definition}

\theoremstyle{remark}
\newtheorem{Rem}[Thm]{Remark}
\newtheorem{Ex}[Thm]{Example}

\numberwithin{equation}{section}
\title[Ellipticity and the theorem of iterates]{Ellipticity and the problem of iterates
in Denjoy-Carleman classes}
\author{Stefan F\"urd\"os}
\address{Instituto de Matem\'atica e Estat\'istica, Universidade de S\~{a}o Paulo, Rua do Mat\~{a}o 1010, 05508-090 S\~{a}o Paulo, SP, Brazil}
\email{stefan.fuerdoes@univie.ac.at}
\thanks{Stefan F\"urd\"os is supported by Austrian Science Fund (FWF) grant J4439}
\author{Gerhard Schindl}
\address{Fakult{\"at} f{\"u}r Mathematik, Universit{\"a}t Wien, Oskar-Morgenstern-Platz 1, 1090 Wien, Austria}
\email{gerhard.schindl@univie.ac.at}
\thanks{Gerhard Schindl is supported by Austrian Science Fund (FWF) project P33417.}

\subjclass[2020]{Primary 35B65, Secondary 26E10, 35J99, 46E10}
\keywords{problem of iterates, ultradifferentiable vectors, strong non-quasianalyticity, non-elliptic operators, optimal functions in Denjoy-Carleman classes}

\begin{document}
	\begin{abstract}
		In 1978 M{\'e}tivier showed that a differential operator $P$ with analytic coefficients
		is elliptic if and only if the theorem of iterates holds for $P$ with respect
		to any non-analytic Gevrey class. In this paper we extend this theorem to Denjoy-Carleman
		classes given by strongly non-quasianalytic weight sequences.
		The proof involves a new way to construct optimal functions in Denjoy-Carleman classes, which might be of independent interest.
		
		Moreover, we point out that the analogous statement for Braun-Meise-Taylor 
		classes given by weight functions
		cannot hold. This signifies an important difference in the properties of Denjoy-Carleman classes and Braun-Meise-Taylor classes, respectively.
	\end{abstract}
\maketitle
\section{Introduction}
In this paper we continue our investigation of the problem of iterates using the modern
theory of ultradifferentiable classes which was started in \cite{Fuerdoes2022}.
There has been recently a lot of work regarding the problem of iterates in
various settings, e.g.\ \cite{MR3380075}, \cite{MR3652556},  \cite{Bouzar2003}, \cite{MR3556261},
\cite{MR2801277}, \cite{MR3043156}, \cite{Chaili2018}, \cite{Derridj2019a}, \cite{MR4098643}
and \cite{Vuckovic2016}.
For a survey of the older literature we refer to \cite{MR557524}.

In a simple form the problem of iterates is the following question: If some derivatives
of a smooth function $u$ satisfy certain uniform estimates, does it follow that all derivatives
of $u$ satisfy these estimates?
The first result in this direction appeared in the literature when in 1962 it was proven by Kotake-Narasimhan \cite{KotakeNarasimhan} and Komatsu \cite{komatsu1962} separately that if $P$ is 
an elliptic differential operator of order $d$ with analytic coefficients 
in some open set $\Omega\subseteq\R^n$ then a smooth function
$u\in\E(\Omega)$ is analytic in $\Omega$  if 
for all compact subsets $K$ of $\Omega$ there is a constant $C>0$ such that
\begin{equation}\label{AnalyticVectors}
	\norm*{P^k u}{L^2(K)}\leq C^{k+1} (dk)!
\end{equation}
for all non-negative integers $k\in\N_0$.
Similar results have been obtained by Nelson \cite{MR107176} for elliptic systems of analytic vector fields.
Functions, which satisfy \eqref{AnalyticVectors}, are usually referred to as analytic vectors of 
the operator $P$.

Generally, given an ultradifferentiable structure\footnote{
	An ultradifferentiable structure is in our context  a sheaf of smooth functions which is defined by estimates on the derivatives of its elements.} $\mathcal{U}$
we say that a smooth function $u$ is an ultradifferentiable vector of class $\mathcal{U}$ 
(or short: a $\mathcal{U}$-vector) of the operator $P$ if the iterates $P^ku$ satisfy the
defining estimates of $\mathcal{U}$.
The problem of iterates for a given operator 
is then the study of the regularity of ultradifferentiable vectors of this operator.
In various different settings it has been proven that if the operator $P$ is elliptic then
the theorem of iterates for $P$ holds, i.e.\
every ultradifferentiable vector of $P$ is an ultradifferentiable function of the same class, see e.g.\ \cite{MR3652556}, \cite{MR548225}, \cite{MR557524}, \cite{MR1875423}, \cite{MR2801277} and \cite{HoepfnerRampazo}.


On the other hand, the problem of iterates for non-elliptic operators was for example considered in
\cite{MR654409}, \cite{MR632764},\cite{MR4098643} and \cite{Fuerdoes2022}.
In particular, M\'{e}tivier \cite{doi:10.1080/03605307808820078} showed that for
any non-elliptic differential operator $P$ with analytic coefficients and any non-analytic 
Gevrey class $\G^s$, i.e.~$s>1$, there is an $s$-Gevrey vector of $P$
 which is not an $s$-Gevrey function.
 The question, which one can now pose, is the following: Can we extend this result 
 to more general classes?
 
In the literature there are mainly two families of ultradifferentiable classes which generalize 
the family of Gevrey classes: First, the Denjoy-Carleman classes given by weight sequences,
see \cite{Komatsu73} and, secondly, the Braun-Meise-Taylor classes given by weight functions,
in the modern form introduced in \cite{MR1052587}. We may note that these classes do not necessarily
coincide according to \cite{BonetMeiseMelikhov07}.

In \cite{Fuerdoes2022} the authors extended M{\'e}tivier's theorem to Denjoy-Carleman classes given by the
sequences $(q^{k^2})_k$, where $q>1$ is a parameter.
On the other hand,  \cite[Theorem 1.1]{Fuerdoes2022} states that the theorem of iterates holds
for analytic-hypoelliptic operators of principal type with respect to Braun-Meise-Taylor classes
given by a certain subfamiliy of weight functions.

In this article, we will prove that M\'{e}tivier's Theorem holds for every strongly non-quasianalytic
weight sequence.
Using the Borel map we can give an invariant version of our result. Recall that the Borel map 
$\lambda_0$
takes every smooth function to the sequence of its derivatives evaluated at the origin.
Clearly the image of any ultradifferentiable structure $\mathcal{U}$ is contained in 
the space $\Lambda_{\mathcal{U}}$ of sequences which satisfy analogous estimates.
We call $\lambda_0:\, \mathcal{U}\rightarrow\Lambda_{\mathcal{U}}$ the Borel map
associated to the structure $\mathcal{U}$.
Applying Petzsche's work \cite{Petzsche1988} on the Borel map associated to Denjoy-Carleman classes 
we obtain the following version of our main theorem:
\begin{Thm}\label{BorelThm}
	If the Borel map associated to the Denjoy-Carleman class $\E^{[\bM]}$ is surjective then
	for any non-elliptic differential operator $P$ with analytic coefficients there is
	an ultradifferentiable vector $u$ of class $[\bM]$ of $P$ which is not an element 
	of $\E^{[\bM]}$.
\end{Thm}
For the notations and definitions used above see the next section and  
for the precise context of Theorem \ref{BorelThm} we refer in particular to 
Remark \ref{BorelRemark}.
Moreover, we need to point out that the analogous statement for Braun-Meise-Taylor classes 
is not true.
In fact, if we consider the subfamily of weight functions from \cite[Theorem 1.1]{Fuerdoes2022}
mentioned above, then we observe that the Borel map associated to the Braun-Meise-Taylor classes defined by these weight functions is surjective, see Remark \ref{BMTRemark} for more details.
This shows a significant difference between Denjoy-Carleman classes and Braun-Meise-Taylor classes 
regarding the problem of iterates.

The proof of our main theorem is a modification of the arguments in \cite{doi:10.1080/03605307808820078}.
In particular, the vector $u$ in the assertion of Theorem \ref{BorelThm}
 is constructed as a Fourier-type integral, which 
 is a smooth function but there is a weight sequence $\bN$, depending only on $\bM$ and $P$
  such that
 $\E^{\{\bN\}}$ is strictly larger than $\E^{\{\bM\}}$ and $u$ is not  an element
 in any Denjoy-Carleman class strictly smaller than $\E^{\{\bN\}}$.
 Moreover, if the original class $\E^{\{\bM\}}$ is closed under derivation we can show
 that $u$ is indeed an element of $\E^{\{\bN\}}$, i.e.\
$u$ is an optimal function of the class $\E^{\{\bN\}}$.
Historically, the construction of optimal functions in Denjoy-Carleman classes was mainly done by using
Fourier series, see e.g.~\cite{MR2384272}.
We believe that this new approach using integrals to construct optimal functions will have various
applications.
It uses an integral kernel, which is very closely related to the so-called $\{\bN\}$-optimal flat
functions introduced in \cite{JimenezGarrido2022a}, which are holomorphic functions defined in a sector of the Riemann
surface of the logarithm whose asymptotic expansion at the origin is the zero series and
the rate of convergence is exactly determined by the weight sequence $\bN$.
We should also note that the existence of such optimal $\{\bN\}$-flat functions, see
\cite{JimenezGarrido2022a}, is closely related
to the non-vanishing of an invariant $\gamma(\bN)$ of the weight sequence $\bN$ which was first introduced by Thilliez \cite{Thilliez2003}.
This invariant itself will also play a significant role in the proof of our main theorem.

The paper is structured in the following way: We present the fundamental definitions of Denjoy-Carleman 
classes and the precise formulation of our main results in  section \ref{MainResults}.
In section \ref{weight} we assemble the technical results we need for the proof of these statements.
The proofs themselves are contained in section \ref{proofs}.
In the last section \ref{OptimalSection} we discuss the optimality of the functions constructed in section \ref{proofs}.
\section{Statement of main results}\label{MainResults}
In this note $\Omega$ will always denote an open subset of the Euclidean space $\R^n$
and $\N$ is the set of positive integers.
\begin{Def}
In our setting a weight sequence is a sequence $\bM=(M_k)_{k\in\N_0}$ of positive numbers
such that $M_0=1\leq M_1$ and the sequence $\mu_k=M_k/M_{k-1}$ is increasing and satisfies
$\lim_{k\rightarrow\infty}\mu_k=\infty$.
\end{Def}
\begin{Rem}
	The fact that the sequence $(\mu_k)_{k\in\N}$ is increasing is equivalent to the logarithmic convexity of the weight sequence $(M_k)_k$, i.e.\
\begin{equation}\label{logconvexity}
	M_k^2\leq M_{k-1}M_{k+1}
\end{equation}
for all $k\in\N$.
Then it follows that the sequence $\sqrt[k]{M_k}$ is increasing and
\begin{equation}\label{Limit2}
	\lim_{k\rightarrow\infty}\sqrt[k]{M_k}=\infty.
\end{equation}
	\end{Rem}

If $K$ is a compact subset of $\R^n$ then $\E(K)$ is the space of smooth functions $f$ for which
there is a neighborhood $U_f$ of $K$ such that $f\in\E(U_f)$.
 If $\bM$ is a weight sequence then we can define two spaces of ultradifferentiable functions on 
 $K$: First, a smooth function $f\in\E(K)$ is an element of the Roumieu class $\Rou{\bM}{K}$
 associated to $\bM$ if
 there are constants $C,h>0$ such that
 \begin{equation}\label{BasicEstimate0}
 	\sup_{x\in K}\,\abs*{D^\alpha f(x)}\leq Ch^\alp M_\alp
 \end{equation}
for all $\alpha\in\N_0^n$. On the other hand a function $u\in\E(K)$ is an element of the Beurling
class $\Beu{\bM}{K}$ if for all $h>0$ there is a constant $C>0$ such that \eqref{BasicEstimate0}
holds for all $\alpha\in\N_0^n$.
We will use the notation $[\bM]=\{\bM\},(\bM)$ throughout the paper to refer
simultaneously to the Roumieu and Beurling case.
If $\Omega\subseteq\R^n$ is an open set then the local classes $\DC{\bM}{\Omega}$
are given in the following way: A smooth function $f\in\E(\Omega)$ is an element of 
$\DC{\bM}{\Omega}$ if for all compact subsets $K$ of $\Omega$ we have that
$f\vert_K\in\DC{\bM}{K}$.
For more details on the functional analytic structure of these spaces we refer to \cite{Komatsu73}.

If $P$ is a differential operator of order $d$
with coefficients in $\Rou{\bM}{\Omega}$ for some weight
sequence $\bM$ then we say that
 a distribution
$u\in\D^\prime(\Omega)$ is an ultradifferentiable Roumieu vector associated to $\bM$ of $P$
(or $u$ is an $\{\bM\}$-vector of $P$) if $P^ku\in L^2_{loc}(\Omega)$ 
for all $k\in\N_0$\footnote{ We use here the convention $P^0=\Id$.}
and for all compact $K\subseteq\Omega$ there are
constants $C,h>0$ such that
\begin{equation}\label{VectorEstimate}
\norm*{P^ku}{L^2(K)}\leq Ch^k M_{dk}
\end{equation}
for all $k\in\N_0$.
On the other hand, if the coefficients of $P$ are in $\Beu{\bM}{\Omega}$ then
 the distribution $u$ is an ultradifferentiable Beurling vector associated
to $\bM$ of $P$ (short: $(\bM)$-vector of $P$) if
$P^ku\in L^2_{loc}(\Omega)$ for every $k\in\N_0$ and for all compact $K\subseteq$ $\Omega$ 
and all $h>0$
there is some constant $C>0$ such that \eqref{VectorEstimate} is satisfied for all $k\in\N_0$.
The space of all $[\bM]$-vectors of $P$ is denoted by $\vDC{\bM}{\Omega}$.
We have always $\DC{\bM}{\Omega}\subseteq\vDC{\bM}{\Omega}$, cf.~\cite{https://doi.org/10.48550/arxiv.2212.11905}.
If the weight sequence $\bM$ satisfies
\begin{equation}\label{AnalInclusion}
	\lim_{k\rightarrow\infty} \sqrt[k]{\frac{M_k}{k!}}=\infty,
\end{equation}
then $\DC{\bM}{\Omega}\subseteq\vDC{\bM}{\Omega}$ for all operators $P$ with analytic coefficients,
since
condition \eqref{AnalInclusion} means in particular 
that the space of analytic functions $\An(\Omega)$ is (strictly) contained in $\DC{\bM}{\Omega}$.

If $P$ is an elliptic differential operator with analytic coefficients 
and the weight sequence $\bM$ satisfies both \eqref{AnalInclusion} and
\begin{equation}\label{DerivClosed}
	\sup_{k\in\N}\sqrt[k]{\frac{M_{k}}{M_{k-1}}}<\infty
\end{equation}
then
we have that the theorem of iterates holds, i.e.~$\vDC{\bM}{\Omega}=\DC{\bM}{\Omega}$.
The Roumieu case was proven in \cite{MR557524}, for the Beurling case see
\cite{Fuerdoes2022}.
We may note that \eqref{DerivClosed} implies that $\DC{\bM}{\Omega}$ is closed under derivation.

Our main theorem gives a converse to this statement.
We say a weight sequence $\bM$ is strongly non-quasianalytic if there is a constant $A$ such that
\begin{equation}\label{StrongQuasi}
	\sum_{k=j+1}^\infty\frac{M_{k-1}}{M_k}\leq A(j+1)\frac{M_j}{M_{j+1}}
\end{equation}
for all $j\in\N_0$.
\begin{Thm}\label{MainThm}
	Let $P$ be a non-elliptic differential operator with analytic coefficients in some open set $\Omega\subseteq\R^n$.
If $\bM$ is a strongly non-quasianalytic weight sequence
then there is a smooth function $u\in\E(\Omega)$ such that 
\begin{equation*}
u\in\vBeu{\bM}{\Omega}\!\setminus\!\Rou{\bM}{\Omega}.
\end{equation*}
\end{Thm}
\begin{Rem}
	We may point out that the proof of Theorem \ref{MainThm} gives actually a slightly stronger result: There is a weight sequence
	$\widetilde{\bM}$ depending on $\bM$ and $P$ such that $u\in\vRou{\widetilde{\bM}}{\Omega}\subseteq\vBeu{\bM}{\Omega}$.
\end{Rem}

Since \eqref{StrongQuasi} implies \eqref{AnalInclusion} we obtain immediately as a corollary of Theorem \ref{MainThm} a generalization of \cite[Theorem 1.2]{doi:10.1080/03605307808820078}:
\begin{Cor}\label{MetivierCor}
 Let $P$ be a differential operator with analytic coefficients in $\Omega$ and 
 $\bM$ be a strongly non-quasianalytic weight sequence which also satisfies \eqref{DerivClosed}.
Then the following statements are equivalent.
\begin{enumerate}
	\item The operator $P$ is elliptic in $\Omega$.
	\item $\vRou{\bM}{\Omega}=\Rou{\bM}{\Omega}$.
	\item $\vBeu{\bM}{\Omega}=\Beu{\bM}{\Omega}$.
\end{enumerate}
\end{Cor}

\begin{Ex}
	\begin{enumerate}
\item
	Let $s\geq 1$. The Gevrey sequence $\bG^s=((k!)^s)_k$ is a strongly non-quasianalytic
	weight sequence if $s>1$.
	Furthermore it is easy to see that $\bG^s$ satisfies \eqref{DerivClosed} for all $s\geq 1$.
	We will denote the Gevrey class of order $s$ by $\G^s(\Omega)=\Rou{\bG^s}{\Omega}$.
 \item Let $q>1$ and $r>1$. The weight sequence $\bN^{q,r}$ given by $N^{q,r}_k=q^{k^r}$
 is strongly non-quasianalytic for every $q>1$ and $r>1$ but \eqref{DerivClosed} holds
 if and only if $1<r\leq 2$.
 \item Let $\sigma> 0$. The weight sequence $\bL^\sigma$ given by
 	$L^\sigma_k=k!(\log(e+k))^{\sigma k}$
is not strongly non-quasianalytic but satisfies \eqref{AnalInclusion} and \eqref{DerivClosed}
for every $\sigma>0$.
	\end{enumerate}
\end{Ex}

\begin{Rem}\label{BorelRemark}
	Let $\Lambda$ be the space of all complex sequences $a=(a_j)_j$ and
	if $\bM$ is a weight sequence then
	let us denote by $\Lambda_{\{\bM\}}$ the space of all complex sequences $a\in\Lambda$
	 for which there are constants $C,h>0$ such that
	\begin{equation}\label{FormalCoeff}
		\abs*{a_j}\leq Ch^jM_j
	\end{equation}
for all $j\in\N_0$.
The set $\Lambda_{(\bM)}$ consists of all $a\in\Lambda$ such that for every $h>0$ there is
some $C>0$ so that \eqref{FormalCoeff} holds for every $j\in\N_0$.

The Borel map  $\lambda_0:\,\E([-1,1])\rightarrow \Lambda$ at the origin is defined by
\begin{equation*}
	\lambda_0(f):=\left(f^{(j)}(0)\right)_{j\in\N_0},\qquad f\in\E([-1,1]).
\end{equation*}
Clearly, $\lambda_0(\DC{\bM}{[-1,1]})\subseteq \Lambda_{[\bM]}$.
	According to \cite{Petzsche1988} the restricted Borel map 
	\begin{equation*}
		\lambda_0\vert_{\DC{\bM}{[-1,1]}}:\DC{\bM}{[-1,1]}\longrightarrow \Lambda_{[\bM]}
	\end{equation*}
is surjective
	 if and only if the weight sequence $\bM$ is
	strongly non-quasianalytic in both the Roumieu and Beurling case.
	Combining this fact with Theorem \ref{MainThm} gives instantly Theorem \ref{BorelThm}.
\end{Rem}
\begin{Rem}\label{BMTRemark}
	If we consider Braun-Meise-Taylor classes $\DC{\omega}{\Omega}$ given by weight functions $\omega$, (for a definition of $\DC{\omega}{\Omega}$ and $\vDC{\omega}{\Omega}$ see e.g.~\cite{Fuerdoes2022}) instead of Denjoy-Carleman classes then the 
	 analogous versions of Theorem \ref{MainThm}, resp.~ Theorem \ref{BorelThm} 
	 (and thus also Corollary \ref{MetivierCor}) cannot hold.
	In fact, according to  \cite[Theorem 1.1]{Fuerdoes2022} we have that
	$\vDC{\omega}{\Omega}=\DC{\omega}{\Omega}$ 
	holds for analytic-hypoelliptic operators $P$ of principal type and 
weight functions $\omega$ for which there is 
	a constant $H>0$ such that
	\begin{equation}\label{omega8}
	\omega\bigl(t^2\bigr)=O(\omega (Ht)) \qquad \text{for } t\rightarrow \infty.
	\end{equation}
On the other hand, according to \cite{Bonet1992} the Borel map associated to $\E^{[\omega]}$ is surjective if and only if
$\omega$ satisfies the strong non-quasianalyticity condition for weight functions:
\begin{equation}\label{omegastrong}
\int_1^\infty\!\frac{\omega(ty)}{t^2}\,dt =O(\omega(y))\qquad \text{for } y\rightarrow \infty.
\end{equation}
It follows from \cite[Lemma A.1]{sectorialextensions1} and \cite[Lemma 4.3]{sectorextensions}
that \eqref{omega8} implies \eqref{omegastrong}.
\end{Rem}

Finally we note that in the proof of Theorem \ref{MainThm} the assumption that the coefficients
of the operator $P$ are analytic is not really necessary. In fact, we only need that
the coefficients are regular enough relative to the weight sequence $\bM$.
However,  the analyticity of the coefficients allows us to write down our main Theorem \ref{MainThm}
in a concise form. 
Nevertheless, in the case of Gevrey classes we give another generalization of 
\cite[Theorem 1.2]{doi:10.1080/03605307808820078}:
\begin{Thm}\label{GevreyThm}
	Let $1\leq r<s$. If $P$ is a differential operator with coefficients in $\G^r(\Omega)$
	for some open set $\Omega\subseteq\R^n$ then the following statements are equivalent:
	\begin{enumerate}
		\item $P$ is elliptic in $\Omega$.
		\item $\G^s(\Omega,P)=\G^s(\Omega)$.
	\end{enumerate}
\end{Thm}
\section{Preliminaries}\label{weight}
\subsection{Weight sequences}
In this section we summarize the facts on weight sequences and the associated spaces which we
will need for the proof of Theorem \ref{MainThm}.
We begin by recalling an auxiliary result from \cite[Section 6D]{Fuerdoes2022}: 

\begin{Lem}[{\cite[Lemma 6.10]{Fuerdoes2022}}]\label{D-AuxLemma1}
	Let $\bM$ be a weight sequence and $\rho, R\geq 1$. Then
	\begin{equation*}
		\rho^jM_{k+l}R^l\leq \rho^{j+l}M_k+M_{j+k+l}R^{j+l}
	\end{equation*}
	for all $j,k,l\in\N_0$.
\end{Lem}
\begin{proof}
	For $\rho\geq \mu_{k+l}R$ we obtain that
	\begin{equation*}
		M_{k+l}R^l=M_k\mu_{k+1}R\dots \mu_{k+l}R\leq \rho^{l}M_k
	\end{equation*}
since $\mu_k$ is increasing.
	If $\rho\leq \mu_{k+l}R$ then
	\begin{equation*}
		\rho^j\leq \mu_{k+l+1}R\dots \mu_{k+l+j}R\leq \frac{M_{j+k+l}}{M_{k+l}}R^j.
	\end{equation*}
\end{proof}
We define on the set of all weight sequences the following order relations: Let $\bM$ and $\bN$ be
two weight sequences then 
\begin{itemize}
	\item $\bM\leq A\bN$ for some $A>0$ if $M_k\leq A N_k$ for all $k\in\N_0$.
	\item $\bM\preceq \bN$ if there are constants $C,h>0$ such that $M_k\leq Ch^k N_k$ for all
	$k\in\N_0$.
	\item $\bM\lhd\bN$ if for every $h>0$ there is a constant $C>0$ such that $M_k\leq Ch^kN_k$
	for all $k\in\N_0$.
\end{itemize}
Moreover, we write $\bM\approx\bN$ if $\bM\preceq\bN$ and $\bN\preceq\bM$.
It is clear that $\bM\preceq\bN$ implies that $\DC{\bM}{\Omega}\subseteq\DC{\bN}{\Omega}$ and
$\vDC{\bM}{\Omega}\subseteq\vDC{\bN}{\Omega}$ for any differential operator $P$ of class $[\bM]$.
If $\bM\lhd\bN$ then $\Rou{\bM}{\Omega}\subseteq\Beu{\bN}{\Omega}$ and
$\vRou{\bM}{\Omega}\subseteq\vBeu{\bN}{\Omega}$.

We denote the space of smooth functions with compact support in $\Omega$ by $\D(\Omega)$.
\begin{Def}
	A subspace $E$ of $\E(\Omega)$ is called quasianalytic if $E\cap\D(\Omega)=\{0\}$, i.e.\
	$E$ contains no non-trivial functions with compact support.
\end{Def}
In the case of Denjoy-Carleman classes the Denjoy-Carleman theorem characterizes the quasianalyticity of the spaces, see e.g.~\cite{MR1996773} or \cite{Komatsu73}:
\begin{Thm}
	Let $\bM$ be a weight sequence. Then the class $\DC{\bM}{\Omega}$ is not quasianalytic 
	if and only if 
	\begin{equation}\label{NonQuasi}
		\sum_{k=0}^\infty\frac{M_k}{M_{k+1}}<\infty.
	\end{equation}
\end{Thm}
We say that a weight sequence $\bM$ is non-quasianalytic if \eqref{NonQuasi} is satisfied and
quasianalytic otherwise.
Clearly \eqref{StrongQuasi} implies \eqref{NonQuasi}.
\subsection{Associated weights}
\begin{Def}
	If $\bM$ is a weight sequence then 
	\begin{align*}
		\omega_\bM(t)&=\sup_{k\in\N_0}\log \frac{t^k}{M_k},\quad t>0,& \omega_\bM(0)&=0,\\
	\shortintertext{is the weight function associated to $\bM$.
The function }
	h_\bM(t)&=\inf_{k\in\N_0}t^kM_k,\qquad t>0,& h_\bM(0)&=0,
	\end{align*}
is the weight associated to $\bM$.
\end{Def}
It is clear that $\omega_\bM$ and $h_\bM$ are well-defined, continuous and increasing functions on 
the right half-line $[0,\infty)$ due to \eqref{logconvexity} and \eqref{Limit2}.
 Moreover it is easy to see that
\begin{equation}
h_\bM\left(t^{-1}\right)=e^{-\omega_\bM(t)}
\end{equation}
for $t>0$.
We can recover the weight sequence from its associated weight function:
\begin{equation}\label{InversionWeight}
	M_k=\sup_{t\geq0}\frac{t^k}{\exp (\omega_\bM(t))},\qquad k\in\N_0.
\end{equation}

To any pair $(\bM,\bN)$ of weight sequences we can define the pointwise product $\bM\bN=\bM\cdot\bN$
by 
\begin{equation*}
	(MN)_k=M_kN_k,\qquad k\in\N_0.
\end{equation*}
If $\tau>0$ then we define similarly the power $\bM^\tau$ of the weight sequence
$\bM$ by $M_k^\tau=(M_k)^\tau$.
Both $\bM\bN$ and $\bM^\tau$ are again weight sequences.
We will need the following result.
\begin{Lem}\label{auxlemma}
Let $\mathbf{T}$ and $\mathbf{U}$ be two weight sequences and $\tau>1$. 
Then the following two assertions
are equivalent:
	\begin{enumerate}
		\item There is a constant $A\geq 1$ such that
		\begin{equation*}
			\mathbf{U}\leq A\bT^\tau.
		\end{equation*}
	\item There is a constant $C\geq 1$ such that
	\begin{equation*}
		\omega_{\mathbf{T}}(s)\leq \tau^{-1}\omega_{\mathbf{U}}\bigl(s^\tau\bigr)+C
	\end{equation*}
for all $s\geq 0$.
\end{enumerate}
If one of the above conditions holds for $\mathbf{T}$, $\mathbf{U}$ and $\tau$ then for all
$0<a<1$, and $\sigma>\tau$ there exists a constant $C\geq 1$ such that
\begin{equation}\label{AuxEstimate3}
	\omega_{\mathbf{T}}(s)\leq \tau^{-1}\omega_{\mathbf{U}}\bigl(as^\sigma\bigr)+C
\end{equation}
for all $s\geq 0$.
\end{Lem}

\begin{proof}
We begin by assuming (2). Since both $\mathbf{T}$ and $\mathbf{U}$ are weight sequences we can apply \eqref{InversionWeight} and obtain for all $k\in\mathbb{N}_0$ that
\begin{align*}
	T_k&=\sup_{s\ge 0}\frac{s^k}{\exp(\omega_{\mathbf{T}}(s))}\ge\frac{1}{e^C}
	\sup_{s\ge 0}\frac{s^k}{\exp(\tau^{-1}\omega_{\mathbf{U}}(s^{\tau}))}
	=\frac{1}{e^C}\sup_{t\ge 0}\frac{t^{k/\tau}}{\exp(\tau^{-1}\omega_{\mathbf{U}}(t))}
	\\&
	=\frac{1}{e^C}\left(\sup_{t\ge 0}\frac{t^{k}}{\exp(\omega_{\mathbf{U}}(t))}\right)^{1/\tau}
	=\frac{1}{e^C}(U_k)^{1/\tau},
\end{align*}
thus
\begin{equation*}
	U_k\le e^{C\tau}(T_k)^{\tau}
\end{equation*}
for all $k\in\N_0$ and (1) is verified with $A:=e^{C\tau}$.

If (1) holds then  let us first recall the following immediate consequence of taking any power $a>0$ of a given weight sequence $\bM$:
\begin{equation*}
	 \omega_{\bM^a}(t)=\sup_{k\in\N_0}\log\left(\frac{t^{k/a}}{M_k}\right)^a
	=a\omega_{\bM}\left(t^{1/a}\right)
\end{equation*}
for all $t\geq 0$.
By assumption we have $t^k/(T_k)^{\tau}\le At^k/U_k$ for all $t\ge 0$ and
 $k\in\N_0$ and thus by the above and the definition of associated weight functions 
 we get for all $t\ge 0$ that
 \begin{equation*}
 	\tau\omega_{\bT}(t^{1/\tau})
 	=\omega_{\bT^{\tau}}(t)\le\omega_{\mathbf{U}}(t)+\log(A).
 \end{equation*}
Thus (2) is verified with the same $\tau$ and $C:=\log A/\tau$.

In order to prove the last statement assume that (2) holds and note that 
for all $\sigma>\tau$ we have $s^{\tau}\le as^{\sigma}$ for all $s$ sufficiently large 
(depending on given $a<1$ and $1<\tau<\sigma$). 
Thus \eqref{AuxEstimate3} is verified if we sufficiently enlarge the constant $C$.
\end{proof}
\subsection{Asymptotic expansions}
Let $\RR$ be the Riemann surface of the logarithm and $\C[[z]]$ be  the space
of formal power series  with complex coefficients.
If $S$ is a sector of $\RR$ we say that a holomorphic function $g\in\hol(S)$ admits 
$\hat{g}=\sum_j a_jz^j\in\C[[z]]$
as its uniform $\{\bM\}$-asymptotic expansion if there are constants $C,A>0$ such that
for all $k\in\N_0$ we have
\begin{equation*}
	\abs*{f(z)-\sum_{j=0}^{k-1}a_jz^j}<CA^kM_k\abs{z}^k,\quad z\in S.
\end{equation*}
We denote the space of functions in $S$ admitting an uniform $\{\bM\}$-asymptotic expansion by
$\fhRou{\bM}{S}$.

We are here mainly interested in holomorphic functions defined on unbounded sectors bisected by the 
real line, i.e.~sectors of the form
\begin{equation*}
	S_\gamma=\Set*{z\in\RR\given \abs{\arg z}<\frac{\gamma\pi}{2}}
\end{equation*}
with opening $\gamma\pi$ where $\gamma>0$.

\begin{Def}[{\cite[Definition 3.3]{JimenezGarrido2022a}}]
	Let $\bM$ be a weight sequence and $S$ be an unbounded sector in $\RR$ bisected by the interval $(0,\infty)$. A holomorphic function $G\in\hol(S)$ is called an optimal $\{\bM\}$-flat
	function in $S$ if 
	\begin{enumerate}
		\item There are constants $A_1$ and $B_1$ such that
		\begin{equation}\label{OptimalEst1}
			A_1h_\bM\left(B_1t\right)\leq G(t)
		\end{equation}
	for all $t>0$.
		\item There exist $A_2,B_2>0$ such that
		\begin{equation}\label{OptimalEst2}
		\abs{G(z)}\leq A_2h_\bM\bigl(B_2\abs{z}\bigr)
		\end{equation}
	for all $z\in S$.
	\end{enumerate}
\end{Def}
If $G$ is an optimal $\{\bM\}$-flat function in a sector $S$ 
then condition (i) in particular gives that $G(t)>0$ for $t>0$.
Moreover, we derive from  \eqref{OptimalEst2} that
\begin{equation*}
	\abs*{G(z)}\leq A_2B_2^k M_k\abs{z}^k, \qquad k\in\N_0,\; z\in S,
\end{equation*}
which implies that $G\in\fhRou{\bM}{S}$ with the null series being the uniform
 asymptotic expansion of $G$
in $S$. On the other hand, the estimate \eqref{OptimalEst1} gives that the rate of decrease
on the real line to $0$ is 
precisely controlled by the sequence $\bM$. This motivates the terminology.

If $G$ is an optimal $\{\bM\}$-flat function in $S$ then we introduce the kernel function
$\Phi=\Phi_\bM:\, S\rightarrow \C$ defined by
\begin{equation*}
	\Phi(z)=G\left(\frac{1}{z}\right).
\end{equation*}
Thus we have the following inequalities:
\begin{gather}\notag 
A_1 e^{-\omega_\bM(t/B_1)}=A_1h_\bM\left(\frac{B_1}{t}\right)
\leq \Phi(t)\\
\shortintertext{for t>0 and } \label{Kernelestimates2}
\abs*{\Phi(z)}\leq A_2h_\bM\left(\frac{B_2}{\abs{z}}\right)
=A_2e^{-\omega_\bM(\abs{z}/B_2)}
\end{gather}
for $z\in S$.
The following proposition is essential for the proof of Theorem \ref{MainThm}.
\begin{Prop}[cf.~{\cite[Proposition 3.11]{JimenezGarrido2022a}}]\label{SanzProp}
	Let $\bM$ be a weight sequence and $G\in\fhRou{\bM}{S}$ be an optimal
	$\{\bM\}$-flat function in a sector $S\subseteq\RR$, which is bisected by the real line
	$(0,\infty)$. If  we put $\Phi(z)=G(1/z)$ then there is a constant $Q_1>0$ such that 
	\begin{equation*}
		Q_1^{k+1}M_k\leq \int_0^\infty\! t^k\Phi(t)\,dt
	\end{equation*}
for all $k\in\N_0$.

Moreover, if $\bM$ satisfies \eqref{DerivClosed} then there exists some $Q_2>0$ such that
\begin{equation*}
	\int_0^\infty\! t^k\Phi(t)\,dt\leq Q^{k+1}_2M_k
\end{equation*}
for all $k\in\N_0$.
\end{Prop}
Proposition \ref{SanzProp} follows directly from the proof of \cite[Proposition 3.11]{JimenezGarrido2022a}.

Finally we need to discuss conditions on the weight sequence $\bM$ which ensure the 
existence of optimal $\{\bM\}$-flat functions.
It turns out to be useful to consider a growth index associated to weight sequences, which
was introduced by Thilliez \cite{Thilliez2003} while investigating the surjectivity of the asymptotic Borel map.
We will use here the characterization given in \cite{JimenezGarrido2019}.
 We recall that a sequence $\bL=(L_k)_k$ is almost increasing if there is some constant $C>0$ such that
\begin{equation*}
	L_j\leq CL_k 
\end{equation*}
for all $j\leq k$.
\begin{Def}
	If $\bM$ is a weight sequence then we define the index $\gamma(\bM)$ of $\bM$ by
	\begin{equation*}
		\gamma(\bM)=\sup\Set*{\gamma>0\given\text{The sequence }
			\biggl(\frac{\mu_k}{k^\gamma}\biggr)_k
 			\text{is almost increasing}}\in [0,\infty].
	\end{equation*}
\end{Def}
It follows directly from the definition that
$\gamma(\bG^s)=s$ for $s\geq 1$ and $\gamma(\bN^{q,r})=\infty$ for all $q,r>1$.
More generally, it is easy to see that if $\bM$ is a weight sequence and $\tau>0$ then
\begin{equation*}
	\gamma\bigl(\bM^\tau\bigr)=\tau\gamma(\bM).
\end{equation*}
We may also note that a weight sequence $\bM$ is strongly non-quasianalytic if and only if
$\gamma(\bM)>1$, see \cite{JimenezGarrido2019}.
Finally, according to \cite[Proposition 3.10]{JimenezGarrido2022a} the following statement is true: If $\bM$ is a weight sequence such that
$\gamma(\bM)>0$ then for each $0<\gamma<\gamma(\bM)$ there exists an optimal $\{\bM\}$-flat function $G$ in the sector $S_\gamma$.

\section{Proof of main theorems}\label{proofs}
In order to prove Theorem \ref{MainThm} we will try to adapt the pattern of the proof of  
\cite[Theorem 2.3]{doi:10.1080/03605307808820078}.
So let 
\begin{equation*}
	P=P(x,D)=\sum_{\alp\leq d} p_\alpha(x)D^\alpha, \quad p_\alpha\in\E(\Omega),
\end{equation*}
be a differential operator of order $d$ 
with smooth coefficients in an open set $\Omega\subseteq\R^n$.
The symbol of $P$ is denoted by
\begin{align*}
	p(x,\xi)&=\sum_{\alp\leq d} p_\alpha(x)\xi^\alpha,\qquad x\in\Omega,\;\xi\in\R^n\!\setminus\!\{0\},\\
	\shortintertext{whereas}
	p_d(x,\xi)&=\sum_{\alp= d} p_\alpha(x)\xi^\alpha,\qquad x\in\Omega,\;\xi\in\R^n\!\setminus\!\{0\},
\end{align*}
is the principal symbol of $P$. 
If we suppose that $P$ is not elliptic then there has to be a point $x_0\in\Omega$ and
a unit vector $\xi_0\in S^{n-1}=\Set{\xi\in\R^n\given \xit=1}$ such that $p_d(x_0,\xi_0)=0$.
Let $\delta>0$ be such that $B_0=\Set{x\in\R^n\given \abs{x-x_0}\leq   2\delta}\subseteq\Omega$.
Then we have the following Lemma.
 \begin{Lem}
		There is a constant $D>0$ such that for all $t>1$, $0<\eps<1$ and all $x\in\Omega$ with 
		$\abs{x-x_0}\leq 2\delta t^{-\eps}$ we have the estimate
		\begin{equation}\label{Condition2}
			\abs*{p(x, t\xi_0)}\leq Dt^{d-\eps}.
		\end{equation}
	\end{Lem}
\begin{proof}
	Since $t>1$ and $\eps>0$ we have that
	\begin{equation*}
		B_t=\Set*{x\in\R^n\given \abs{x-x_0}\leq 2\delta t^{-\eps}}\subseteq B_0.
	\end{equation*}
We are going to write
\begin{equation*}
	p_j(x,\xi)=\sum_{\alp=j}p_\alpha(x)\xi^\alpha
\end{equation*}
for the homogeneous part of order $j\in\Set{0,\dotsc,d-1}$ of $p$ and estimate
\begin{equation*}
	\abs*{p_j(x,t\xi_0)}\leq t^{j}\sup_{x\in B_0}\abs*{p_j(x,\xi_0)}
	\leq t^{d-\eps}\sup_{\substack{x\in B_0\\ j=0,1,\dotsc,d-1}}\abs*{p_j(x,\xi_0)}
\end{equation*}
for $x\in B_t$, $t>1$ and $0<\eps<1$. In the case of the principal symbol $p_d$ we have that
\begin{equation*}
	p_d(x,t\xi_0)=t^dp_d(x,\xi_0).
\end{equation*}
Since $p_d(x_0,\xi_0)=0$ it follows from \cite[Lemma 1.2]{Bierstone1980} 
that there are continuous functions $g_k$, $k=1,\dotsc,n$, in $B_0$ such that
\begin{equation*}
	p_d(x,\xi_0)=\sum_{k=1}^n (x-x_0)^{e_k}g_k(x)
\end{equation*}
where $e_k$ denotes the $k$-th unit vector.
If $x\in B_t$, $t>1$ and $0<\eps<1$ we can thus estimate
\begin{equation*}
	\begin{split}
		\abs*{p_d(x,t\xi_0)}&\leq t^d \abs{x-x_0} \sum_{k=1}^n\abs*{g_k(x)}\\
		&\leq 2\delta t^{d-\eps}n\sup_{\substack{x\in B_0\\ k=1,\dotsc,n}}\abs*{g_k(x)}.
	\end{split}
\end{equation*}
\end{proof}
After this technical observation we begin with the construction of the $\{\bM\}$-vector $u$
for a given weight sequence $\bM$.
We may suppose that there is a non-quasianalytic weight sequence $\bL$ such that
$\bL\preceq \bM$.
Then there exists $\psi\in\Rou{\bL}{\R^n}$ such that 
\begin{equation}\label{TestFct}
	\psi(x)=1\quad \text{for}\; \abs{x}<\delta\qquad \text{\&}\qquad \psi(x)=0\quad 
	\text{for}\; \abs{x}>2\delta.
\end{equation} 
We  define the function $u$ by
\begin{equation}\label{FctDef}
	u(x)=\int_1^\infty\negthickspace \psi\left(t^\eps(x-x_0)\right) \Phi_\bN(t)e^{it\xi_0(x-x_0)}\,dt,\qquad 
	x\in\Omega,
\end{equation}
where $0<\eps<1$ is a parameter
to be specified and $\Phi_\bN(t)=G_\bN(1/t)$ with $G_\bN$ being an optimal $\{\bN\}$-flat function
for some weight sequence $\bN$.
Obviously $u\in\E(\Omega)$ with $\supp u\subseteq B_{0}$ and we observe that if $D_{\xi_0}=-i\partial_{\xi_0}$ then 
\begin{equation*}
	D_{\xi_0}^k u(x_0)=\int_{1}^{\infty}\negthickspace t^{k}\Phi_\bN(t)\,dt.
\end{equation*}
Since $h_\bN(s)\leq 1$ for all $s>0$ we have that
\begin{equation*}
	\begin{split}
		0<	\int_0^1\!t^k \Phi_\bN(t)\,dt\leq C\int_0^1\!t^k\,dt= \frac{C}{k+1}
		\xrightarrow{\;k\rightarrow \infty\;} 0
	\end{split}
\end{equation*}
for some constant $C>0$.
Hence using Proposition \ref{SanzProp} we can conclude that there is a constant $Q_1>0$ such that
\begin{equation}\label{OptimalEstimate}
	Q_1^{k+1} N_k-\frac{C}{k+1}\leq	\abs*{D_{\xi_0}^ku(x_0)} 
\end{equation}
for all $k\in\N_0$.
It follows that 
 $u$ cannot be of class $\{\bM\}$ in any neighborhood
of $x_0$ if $\bM\precnapprox\bN$.

In order to estimate $P^ku$ we need to assume some a-priori regularity on the coefficients
of $P$. Therefore we will suppose that $p_\alpha\in \Rou{\bL}{\Omega}$ for all $\alpha$, 
which means that there is a constant $C_P>0$ such that
for all $\nu\in\N_0^n$, all $\alpha\in\N_0^n$ with $\alp\leq d$, all $x\in B_0$ and every $t\geq 1$
we have that 
\begin{equation}\label{Condition1}
		\abs*{D^\nu_x \partial_\xi^\alpha p\bigl(x,t\xi_0\bigr)}
		\leq C_P^{\nut+1}L_{\nut} t^{d-\alp}.
\end{equation}
We recall also that since $\psi\in\D^{\{\bL\}}(\R^n)=\Rou{\bL}{\R^n}\cap\D(\R^n)$, there are constants $C_0$ and
$h_0$ such that for $\nu\in\N_0^n$
\begin{equation}\label{PsiEstimate}
	\abs*{D^\nu\psi(y)}\leq C_0h_0^{\nut}L_{\nut}
\end{equation}
for all $y\in\R^n$.
Without loss of generality we will assume that $h_0\geq 2C_P$.

Now, if we compute $P^ku$, $k\in\N_0$, we see that
\begin{equation}\label{VectorIntegral}
	P^ku(x)=\int_1^\infty\negthickspace Q_k(x,t)\Phi_\bN(t)e^{i\xi_0(x-x_0)}\,dt
\end{equation}
where the functions $Q_k$ are iteratively defined by
\begin{subequations}\label{Iteration}
\begin{align}
	Q_0(x,t)&= \psi\bigl(t^{\eps}(x-x_0)\bigr)\\
	\shortintertext{and}
	Q_{k+1}(x,t)&=\sum_{\alp\leq d}\frac{1}{\alpha!}\partial_\xi^\alpha p(x,t\xi_0)
	D^\alpha_x Q_k(x,t),\qquad k\in\N_0.\label{Iteration1}
\end{align}
\end{subequations}
We will need the following Lemma:
\begin{Lem}\label{IteratesLemma}
	There exists a constant $A>0$ such that for all $k\in\N_0$, all $\nu\in\N_0^n$, 
	all $x\in B_{0}$ and all $t\geq 1$:
	\begin{equation}\label{AuxEstimate}
		\abs*{D_x^\nu Q_k(x,t)}\leq C_0\left(2h_0t^\eps\right)^{\nut}A ^k
		\left(t^{(d-\eps)k}L_{\nut}+t^{k\eps(2d-1)}L_{\nut+dk}\right).
	\end{equation}
\end{Lem}
\begin{proof}
	We prove the statement by induction in $k$ where the induction hypothesis for step $k$ is
	that \eqref{AuxEstimate} is satisfied for all $\nu$, $x$ and $t$.
	For $k=0$ it is clear that \eqref{AuxEstimate} follows from \eqref{PsiEstimate}.
	Assuming that \eqref{AuxEstimate} is true at step $k$ we will show that \eqref{AuxEstimate}
	is still true at step $k+1$ if we choose the constant $A$ suitable.
	For simplicity, we set
	\begin{align*}
		\rho&=t^{1-\eps/d},\\
		R&=t^{\eps(2-1/d)}.
	\end{align*}
	Observe that $\rho^d=t^{d-\eps}$ and that according to Lemma \ref{D-AuxLemma1} we have that
	\begin{equation*}
		\rho^d L_{\nut+dk} R^{dk}\leq \rho^{d(k+1)}L_{\nut}+L_{\nut+d(k+1)}R^{d(k+1)}.
	\end{equation*}
	If we furthermore put
	\begin{equation*}
		\Lambda(k,\nu)= \rho^{dk}L_{\nut}+R^{dk}L_{\nut+dk}
	\end{equation*}
	then we conclude that
	\begin{gather}\label{Condition3}
		\begin{aligned}
			t^{d-\eps}\Lambda(k,\nu)&\leq \rho^{d(k+1)}L_{\nut}+\rho^dL_{\nut+dk}R^{dk}\\
			&\leq 2\Lambda(k+1,\nu)
		\end{aligned}\\
		\shortintertext{and for all $\alp\leq d$, applying again Lemma \ref{D-AuxLemma1},
		we obtain}
		\begin{aligned}\label{Condition4}
			\rho^{d-\alp}R^{\alp}\Lambda(k,\nu+\alpha)&\leq \rho^{d(k+1)-\alp}L_{\nut+\alp}R^\alp
			+\rho^{d-\alp}L_{\nut+\alp+dk}R^{\alp+dk}\\
			&\leq 2\rho^{d(k+1)}L_{\nut}+2R^{d(k+1)}L_{\nut +d(k+1)}\\
			&\leq 2\Lambda(k+1,\nu).
		\end{aligned}
	\end{gather}
	
	We continue by differentiating \eqref{Iteration1}:
	\begin{equation*}
		D_x^\nu Q_{k+1}(x,t)=\sum_{\alp\leq d}\sum_{\nu^\prime\leq\nu}\frac{1}{\alpha !}\binom{\nu}{\nu^\prime}
		\left(D_x^{\nu-\nu^\prime}\partial^\alpha_{\xi}p\right)(x,t\xi_0)D^{\nu^\prime+\alpha}_xQ_k(x,t).
	\end{equation*}
	We estimate $\abs{D_x^\nu Q_{k+1}(x,t)}$ by $I_1$, $I_2$ and $I_3$ where:
	\begin{align*}
		I_1&=\abs*{p(x,t\xi_0)}\abs*{D_x^\nu Q_k (x,t)}\\
		I_2&=\sum_{\nu^\prime<\nu}\binom{\nu}{\nu^\prime}\abs*{D_x^{\nu-\nu^\prime}
			p(x,t\xi_0)}\abs*{D_x^{\nu^\prime} Q_k(x,t)}\\
		I_3&=\sum_{0<\alp\leq d}\sum_{\nu^\prime\leq\nu}\frac{1}{\alpha !}\binom{\nu}{\nu^\prime}
		\abs*{D_x^{\nu-\nu^\prime}\partial_\xi^\alpha p(x,t\xi_0)}
		\abs*{D_x^{\nu^\prime+\alpha} Q_k(x,t)}
	\end{align*}
	Now observe that if $t$ is fixed the support of $Q_k(\,.\,,t)$ is contained in
	the set $B_t=\Set{x\given\abs{x-x_0}\leq 2\delta t^{-\eps}}$ for all $k$.
	Thence by utilizing the induction hypothesis and \eqref{Condition2} we obtain that
	\begin{align}\notag
		I_1&\leq Dt^{d-\eps}C_0\left(h_0t^\eps\right)^{\nut}\Lambda(k,\nu)	A^k\\
		\shortintertext{and thus \eqref{Condition3} implies that}
		I_1&\leq C_0\left(h_0t^\eps\right)^{\nut}2D\Lambda(k+1,\nu)A^k.\label{Term1}
	\end{align}
	Similarly, by \eqref{Condition1} and the induction hypothesis we have that
	\begin{equation*}
		I_2	\leq A^k\sum_{\nu^\prime<\nu}\binom{\nut}{\nupt}
		C_P^{\abs{\nu-\nu^\prime} +1}t^d L_{\abs{\nu-\nu^\prime}} 
		C_0\left(h_0t^\eps\right)^{\nupt}\Lambda\left(k,\nu^\prime\right).
	\end{equation*}
	By writing $t^d=t^\eps t^{d-\eps}$ and again using \eqref{Condition3} we obtain
	\begin{equation*}
		I_2\leq A^k\sum_{\nu^\prime<\nu}\binom{\nut}{\nupt}C_ P^{\abs{\nu-\nu^\prime}+1}t^\eps L_{\abs{\nu-\nu^\prime}} C_0
		\left(h_0t^\eps\right)^{\nupt}2\Lambda\left(k+1,\nu^\prime\right).
	\end{equation*}
Now, since $\bL$ is a weigth sequence, we have the following estimate:
\begin{equation}\label{Condition5}
	L_{\abs{\nu-\nu^\prime}}\Lambda\left(k+1,\nu^\prime\right)\leq \Lambda(k+1,\nu).
\end{equation}
	Thence
	\begin{equation*}
		I_2\leq \sum_{\nu^\prime<\nu}2^{\nut}
		\left(\frac{C_P}{h_0t^\eps}\right)^{\abs{\nu-\nu^\prime}}
		 t^\eps C_0C_P
		\left(h_0t^\eps\right)^{\nut}2\Lambda(k+1,\nu)A^k.
	\end{equation*}
	If we recall that we have chosen $h_0\geq 2C_P$ and $t\geq 1$ and set 
	\begin{equation*}
		E=\sum_{\alpha\geq 0}\frac{1}{2^\alp}
	\end{equation*}
	then
	\begin{equation}\label{Term2}
		I_2 \leq \frac{2C_P}{h_0}EC_0C_P\left(2h_0t^\eps\right)^{\nut}\Lambda(k+1,\nu) A^k.
	\end{equation}
	Finally, according to \eqref{Condition1} and the induction hypothesis we can estimate $I_3$ by
	\begin{equation*}
		I_3\leq \sum_{0<\alp\leq d}\sum_{\nu^\prime\leq\nu}\binom{\nut}{\nupt}
		C_P^{\abs{\nu-\nu^\prime}+1}t^{d-\alp}C_0\left(h_0t^\eps\right)^{\nupt+\alp}
		\Lambda(k,\nu^\prime+\alpha)A^k
	\end{equation*}
	By the definition of $\rho$ and $R$ we have for all $\alpha\neq 0$ that
	\begin{equation*}
		t^{d-\alp+\eps\alp}\leq\rho^{d-\alp}R^{\alp}
	\end{equation*}
	which together with \eqref{Condition4} implies that
	\begin{equation*}
		t^{d-\alp+\eps\alp}\Lambda\left(k,\nu^\prime+\alpha\right)\leq 2\Lambda\left(k+1,\nu^\prime\right).
	\end{equation*}
	Thus with \eqref{Condition5} we obtain that
	\begin{equation*}
		I_3\leq\sum_{0<\alp\leq d}\sum_{\nu^\prime\leq \nu}
		\left(\frac{C_P}{h_0t^\eps}\right)^{\abs{\nu-\nu^\prime}}
		C_PC_0h_0^\alp\left(2h_0t^\eps\right)^{\nut}2\Lambda(k+1,\nu)A^k.
	\end{equation*}
	Setting $h_1=\sum_{\alp\leq d}h_0^\alp$ it follows that
	\begin{equation}\label{Term3}
		I_3\leq 2Eh_1C_PC_0\left(2h_0t^\eps\right)^{\nut}\Lambda(k+1,\nu)A^{k+1}.
	\end{equation}
	Combining \eqref{Term1}, \eqref{Term2} and \eqref{Term3} we see that
	\begin{equation*}
		I_1+I_2+I_3\leq C_0\left(2h_0t^\eps\right)^{\nut}\Lambda(k+1,\nu)A^{k+1}
	\end{equation*}
	if we choose 
	\begin{equation*}
		A\geq 2D+\frac{2C_PE}{h_0}+2Eh_1C_P.
	\end{equation*}
\end{proof}
If we set $\nu=0$ in \eqref{AuxEstimate} we obtain that
\begin{equation}\label{VectorEstimate1}
	\abs*{Q_k(x,t)}\leq C_0A^k\left(\rho^{dk}+R^{dk}L_{dk}\right)
\end{equation}
where we recall that $\rho=t^{1-\eps/d}$ and $R=t^{\eps(2-1/d)}$.
By the definition of $\omega_\bM$ we have that
\begin{equation*}
	\rho^{dk}\leq  M_{dk}\exp\left(\omega_\bM(\rho)\right).
\end{equation*}
The second term on the right-hand side of \eqref{VectorEstimate1} must be dealt with differently.
If we assume that there is a weight sequence $\bV$ such that $\bL\bV\preceq\bM$ then
using
\begin{equation*}
	R^{dk}\leq V_{dk}\exp\left(\omega_{\bV}(R)\right),
\end{equation*}
we conclude that there are constants $C, h>0$ such that
\begin{equation*}
	\abs*{Q_k(x,t)}\leq Ch^k M_{dk}\left(\exp\left(\omega_\bM\left(\rho\right)\right)+
	\exp\left(\omega_\bV\left(R\right)\right)\right).
\end{equation*}
Now we recall from \eqref{Kernelestimates2} that there are constants $A_2,B_2\geq 1$ such that
\begin{equation*}
	\abs*{\Phi_\bN(t)}\leq A_2\exp\left(-\omega_\bN\left(\frac{t}{B_2}\right)\right)
\end{equation*}
for $t>0$.
Combining these estimates we obtain from \eqref{VectorIntegral} that
\begin{equation}\label{VectorEstimate2}
\abs*{P^ku(x)}\leq A_2Ch^k M_{dk}
\int_1^\infty\negthickspace\exp\left(-\omega_\bN\left(\frac{t}{B_2}\right)\right)
\Biggl(\exp\left(\omega_\bM\left(t^{1-\eps/d}\right) \right) +
\exp\left(\omega_\bV\left(t^{\eps(2-1/d)}\right)\right)\Biggr)\,dt
\end{equation}
for all $x\in B_0$ and all $k\in\N_0$.
If we could choose $\bN$, $\bV$ and $\eps$ in such a way that the integral in \eqref{VectorEstimate2}
converges then we would have shown that $u\in\vRou{\bM}{\Omega}$, since $\supp u\subseteq B_{0}$.

As a first abstract result we note:

\begin{Thm}\label{TechnicalTheorem}
	Let $\bL$, $\bM$, $\bN$, $\bV$ be four weight sequences which satisfy the following properties
	for some $d\in\N$:
\begin{itemize}
	\item $\bM\precnapprox \bN$ and $\gamma(\bN)>0$.
	\item $\bL$ is non-quasianalytic.
	\item $\bV\leq\bM$ and $\bL\bV\preceq\bM$.
	\item There are constants $1<\tau<2d/(2d-1)$ and $A>1$ such that $\bN\leq A\bV^{\tau}$.
\end{itemize}

If $P$ is a non-elliptic linear differential operator of order $d$
 with coefficients in $\Rou{\bL}{\Omega}$,
then there is a smooth function $u\in\E(\Omega)$ such that 
$u\in\vRou{\bM}{\Omega}$
and $u\notin\Rou{\bT}{\Omega}$ for any weight sequence $\bT\precnapprox\bN$. 
Thus in particular $u\notin\Rou{\bM}{\Omega}$.
\end{Thm}
\begin{proof}
	There is an optimal $\{\bN\}$-flat function $G_\bN$ since $\gamma(\bN)>0$.
	Thus we can define the function $u$ by \eqref{FctDef}, i.e.
	\begin{equation*}
		u(x)=\int_{1}^{\infty}\negthickspace\psi\bigl(t^\eps(x-x_0)\bigr)\Phi_\bN(t)
		e^{it\xi_0(x-x_0)}\,dt
	\end{equation*}
where $\Phi_\bN(t)=G_\bN(1/t)$, $(x_0,\xi_0)$ is a non-elliptic point of $P$,
$\psi\in\D_{\{\bL\}}(\R^n)$ is a function with suitable compact support as above  and
$0<\eps<1$ is a parameter to be determined.
The estimate \eqref{OptimalEstimate} gives that $u\notin\Rou{\bT}{\Omega}$ for any
weight sequence
$\bT\precnapprox\bN$.

In order to estimate the iterates $P^ku$ we  apply Lemma \ref{IteratesLemma}
and since $\bL\bV\preceq\bM$ we obtain \eqref{VectorEstimate2}.
 If $\eps\leq 1/2$ then $t^{\eps(2-1/d)}\leq t^{1-\eps/d}$ for all $t\geq 1$. 
 Hence $\omega_\bV(t^{\eps(2-1/d)})\leq \omega_\bV(t^{1-\eps/d})$.
 On the other hand, due to $\bV\leq \bM$ we have $\omega_\bM(s)\leq \omega_\bV(s)$ for all $s\geq 0$.
 Thus, in summary there are constants $C,h>0$ and a constant $B_2>1$ such that 
 \begin{equation*}
 	\abs*{P^ku(x)}\leq Ch^kM_{dk} \int_1^\infty\negthickspace
 	\exp\left(-\omega_\bN\left(\frac{t}{B_2}\right)+\omega_\bV\left(t^{1-\eps/d}\right) \right)
 \,dt.
 \end{equation*}

By assumption there are $A\geq 1$ and $1<\tau<2d/(2d-1)$ such that $\bN\leq A\bV^\tau$.
We choose $\eps\leq 1/2$ such that
\begin{equation*}
	\tau<\frac{d}{d-\eps}<\frac{2d}{2d-1}.
\end{equation*}
Thus we are able to apply
 Lemma \ref{auxlemma} and according to \eqref{AuxEstimate3} there is a constant $\widetilde{C}>0$ such that 
\begin{equation*}
	\omega_\bV(s)\leq \tau^{-1}\omega_\bN\left(\frac{s^{d/(d-\eps)}}{B_2}\right)+\widetilde{C}
\end{equation*}
for all $s\geq 0$. If we set $s=t^{1-\eps/d}$ then we obtain
\begin{equation*}
	-\omega_\bN\left(\frac{t}{B_2}\right)+\omega_\bV\left(t^{1-\eps/d}\right)
	\leq -\left(1-\tau^{-1}\right)\omega_\bN\left(\frac{t}{B_2}\right)+\widetilde{C}.
\end{equation*}
It follows that there are constant $C,h>0$ such that
\begin{equation*}
	\norm*{P^ku}{L^2(B_{2\delta})}\leq Ch^k M_{dk}
	\int_1^\infty\negthickspace
	\exp\left(-\bigl(1-\tau^{-1}\bigr)\omega_\bN\bigl(t/B_2\bigr) \right)\,dt.
\end{equation*}
The integral on the right-hand side of the above estimate converges since $(1-\tau^{-1})>0$
and $\omega_\bN(s)$ increases faster then $\log s^p$ for any $p\in\N$ when $s\rightarrow\infty$,
see \cite{Komatsu73}.
Therefore $u\in\vRou{\bM}{\Omega}$ because $\supp u\subseteq B_{0}$.
\end{proof}

\begin{Cor}\label{GammaInfinity}
	Let $\bM$ be a weight sequence such that $\gamma(\bM)=\infty$ and $\bT$ be a weight sequence
	satisfying $\bT\preceq \bM^\rho$ for all $\rho>0$.
	If $P$ is a non-elliptic differential operator with coefficients in $\Rou{\bT}{\Omega}$
	then there is a smooth function $u$ such that 
	\begin{equation*}
		u\in\vBeu{\bM}{\Omega}\!\setminus\!\Rou{\bM}{\Omega}.
	\end{equation*}
\end{Cor}
\begin{proof}
Let $d$ be the order of $P$. We choose real parameters $0<q,\sigma<1$ and $\rho>1$ such that 
\begin{equation*}
	1<\rho<\frac{2d}{2d-1}\sigma\quad \text{and}\quad 1<\rho q.
\end{equation*}
	We set $\widetilde{\bM}=\bM^q$,
	$\bL=\widetilde{\bM}^{1-\sigma}=\bM^{q(1-\sigma)}$,
	$\bV=\widetilde{\bM}^{\sigma}=\bM^{q\sigma}$ and $\bN=\widetilde{\bM}^{\rho}=\bM^{q\rho}$.
	Then $\bT\preceq\bL\lhd\widetilde{\bM}\lhd\bM\lhd\bN$. 
	Moreover $\gamma(\bL)=\infty$ and thus $\bL$ is, in particular, non-quasianalytic.
	Finally $\bV\leq \bM$ and obviously $\bL\bV\preceq\widetilde{\bM}$.
	It follows also that $\gamma(\bN)>0$ and
		$\bN=\bV^{\tau}$
where $\tau=\rho/\sigma$.

Hence we can apply Theorem \ref{TechnicalTheorem} and infer the existence of a function 
$u\in\E(\Omega)$ such that $u\in\vRou{\widetilde{\bM}}{\Omega}\subseteq\vBeu{\bM}{\Omega}$
and $u\notin\Rou{\bM}{\Omega}$ since $\widetilde{\bM}\lhd\bM\lhd\bN$.
\end{proof}
In order to finish the proof of Theorem \ref{MainThm} it remains to  consider the remaining case  
when $\gamma(\bM)>1$ is finite. For this it turns out to be convenient to follow
the original proof of Metivier in the Gevrey case more closely.
\begin{Thm}\label{FiniteVersion}
	Let $\bM$ be a weight sequence such that $1<\gamma(\bM)<\infty$.
	If $P$ is a non-elliptic differential operator of class $\{\bM^\rho\}$, 
	for some $1<1/\rho<\gamma(\bM)$, in $\Omega$
	then there is a smooth function $u\in\E(\Omega)$ such that
	\begin{equation*}
		u\in\vBeu{\bM}{\Omega}\!\setminus\!\Rou{\bM}{\Omega}.
	\end{equation*}
\end{Thm}
\begin{proof}
	We define a new weight sequence $\bT$ by setting 
	\begin{equation*}
		T_k=M_k^{1/\gamma},\quad k\in\N_0,
	\end{equation*}
	where $\gamma=\gamma(\bM)$. Note that $\gamma(\bT)=1$.
	We can define a scale, i.e.\ an one-parameter family, of weight sequences by 
	$\bT^\sigma$, $\sigma>1$. Then $\bT^\gamma=\bM$ and $\bT^{\rho\gamma}=\bM^\rho$.
	We will denote the weight function associated to $\bT^\sigma$ by $\omega_\sigma$ and
	the weight function associated to $\bT=\bT^1$ by $\omega=\omega_1$.
	Clearly
	\begin{equation}\label{omegas}
		\omega_\sigma(s)=\sigma\omega\bigl(s^{1/\sigma}\bigr)
	\end{equation}
	for all $s\geq 0$.
	Now choose positive numbers $\gamma_0$ and $\widetilde{\gamma}$ such that  $\rho\gamma<\gamma_0<\widetilde{\gamma}<\gamma$
	and 
	\begin{equation*}
		\frac{\gamma-\gamma_0}{2d}>\gamma-\widetilde{\gamma},
	\end{equation*}
	where $d$ denotes the order of $P$.
	Thus $\bT^{\gamma_0}$ is a strongly non-quasianalytic weight sequence since $\gamma(\bT^{\gamma_0})=\gamma_0>\rho\gamma >1$.
	In particular $\bT^{\gamma_0}$ is  non-quasianalytic and obviously
	$\bT^{\gamma_0}\bT^{\widetilde{\gamma}-\gamma_0}= \bT^{\widetilde{\gamma}}$. 
	We may write $\bL=\bT^{\gamma_0}$, $\bV=\bT^{\widetilde{\gamma}-\gamma_0}$ 
	and $\widetilde{\bM}=\bT^{\widetilde{\gamma}}$.
	We set also 
	\begin{align*}
		\eps&=\frac{d(\widetilde{\gamma}-\gamma_0)}{2d\widetilde{\gamma}-\gamma_0}<\frac{1}{2},\\
		\gamma^\prime&=\frac{d\widetilde{\gamma}}{d-\eps}=\frac{2d\widetilde{\gamma}-\gamma_0}{2d-1}.
	\end{align*}
Observe that $\gamma<\gamma^\prime$.
	We put $\bN=\bT^{\gamma^\prime}$ and note that 
	$\widetilde{\bM}\lhd\bM\lhd\bN$.
	Since $\gamma(\bN)=\gamma^\prime>0$
 we can define $u$ by \eqref{FctDef} using the above choices of $\bL$, $\bN$ and $\eps$.
 Then
	it follows from \eqref{OptimalEstimate} that $u$ cannot be of class $\{\bM\}$ in $\Omega$ since $\bM\lhd\bN$.
	
	We recall that
	\begin{equation*}
		P^ku(x)=\int_1^\infty\negthickspace Q_k(x,t)\Phi_{\bN}(t)e^{it\xi_0(x-x_0)}\,dt
	\end{equation*} 
	where the functions $Q_k$ are defined by \eqref{Iteration}.
	If we want to estimate $\abs{P^ku(x)}$ 
	then we observe that
	\begin{equation}\label{FirstStep}
		\abs*{P^ku(x)}\leq C_2\int_1^\infty\negthickspace \abs*{Q_k(x,t)}e^{-\omega_\bN(t/B_2)}\,dt
	\end{equation}
	for some constants $C_2>0$ and $B_2>1$. 
	According to Lemma \ref{IteratesLemma} 
	there are constants $C,h>0$ such that
		\begin{equation*}
			\abs*{Q_k(x,t)}\leq Ch^k\left(t^{(d-\eps)k}+t^{k\eps(2d-1)}L_{dk}\right).
		\end{equation*}
	If we set 
	\begin{align*}
		\rho&=t^{1-\eps/d} & &\text{and} & R&=t^{\eps(2-1/d)}\\
		\intertext{then}
		\rho^{dk}&=t^{(d-\eps)k} & &\text{and} & R^{dk}&=t^{k\eps(2d-1)}.
	\end{align*}
	Observe now that
	\begin{align*}
		\frac{\rho^{dk}}{q^{dk}}&\leq \widetilde{M}_{dk}\exp\left(\omega_{\widetilde{\bM}}\left(\frac{\rho}{q}\right)\right)\\
		\shortintertext{and}
		\frac{R^{dk}}{q_1^{dk}}&\leq V_{dk}
		\exp\left(\omega_{\bV}\left(\frac{R}{q_1}\right)\right)
	\end{align*}
	for any $q,q_1>0$. We set $q=B_2^{(d-\eps)/d}$ and $q_1=B_2^{\eps(2d-1)/d}$
	where $B_2$ is the constant in the exponential in \eqref{FirstStep}.
	We have $q>q_1$ since $\eps<1/2$ and  $B_2>1$.
	Moreover we infer from \eqref{omegas} that 
	\begin{align*}
		\omega_{\widetilde{\bM}}(s)
		&=\widetilde{\gamma}\omega\left(s^{1/\widetilde{\gamma}}\right),\\
		\omega_\bV(s)&=\gamma_1\omega\left(s^{1/\gamma_1}\right),\\
		\omega_\bN(s)&=\gamma^\prime\omega\left(s^{1/\gamma^\prime}\right)
	\end{align*}
	for $s\geq 0$. Here $\omega$ is the associated weight function to $\bT=\bM^{1/\gamma}$ and
	$\gamma_1=\widetilde{\gamma}-\gamma_0$.
	In particular
	\begin{align*}
		\omega_{\widetilde{\bM}}\left(\frac{\rho}{q}\right)&=\widetilde{\gamma} \omega\left(
		\frac{\left(t^{(d-\eps)/d}\right)^{1/\widetilde{\gamma}}}{
			\left(B_2^{(d-\eps)/d}\right)^{1/\widetilde{\gamma}}}\right)
		=\widetilde{\gamma}\omega\left(\frac{t^{1/\gamma^\prime}}{B_2^{1/\gamma^\prime}}\right),\\
		\omega_\bV\left(\frac{R}{q_1}\right)&=\gamma_1\omega\left(
		\frac{\left(t^{\eps(2d-1)/d}\right)^{1/\gamma_1}}{\left(B_2^{\eps(2d-1)/d}\right)^{1/
				\gamma_1}}\right)
		=\gamma_1\omega\left(\frac{t^{1/\gamma^\prime}}{B_2^{1/\gamma^\prime}}\right)\\
		\shortintertext{since $\gamma_1d/(\eps(2d-1))=\gamma^\prime$. Finally}
		\omega_\bN\left(\frac{t}{B_2}\right)&=\gamma^\prime\omega\left(
		\frac{t^{1/\gamma^\prime}}{B_2^{1/\gamma^\prime}}\right).
	\end{align*}
	Thus we can conclude that $\omega_\bV(R/q_1)\leq\omega_{\widetilde{\bM}}(\rho/q)$ 
	since $\gamma_1<\tilde{\gamma}$.
	Therefore we have the following estimate:
	\begin{equation*}
		\abs*{Q_k(x,t)}\leq 2C \left(q^dh\right)^k \widetilde{M}_{dk}\exp\left(
		\omega_{\widetilde{\bM}}\left(\frac{\rho}{q}\right)\right).
	\end{equation*}
	It follows that
	\begin{equation}\label{FinalEstimate}
		\abs*{P^ku(x)}\leq 2CC_2\left(q^dh\right)^k\widetilde{M}_{dk}\int_1^\infty\negthickspace
		\exp\left(-\omega_\bN\left(\frac{t}{B_2}\right)
		+\omega_{\widetilde{\bM}}\left(\frac{\rho}{q}\right)\right)\,dt.
	\end{equation}
	However, our arguments above infer that
	\begin{equation*}
		-\omega_\bN\left(\frac{t}{B_2}\right)+\omega_{\widetilde{\bM}}\left(\frac{\rho}{q}\right)
		=-\gamma^\prime \omega\left(\frac{t^{1/\gamma^\prime}}{B_2^{1/\gamma^\prime}}\right)
		+\widetilde{\gamma}\omega\left(\frac{t^{1/\gamma^\prime}}{B_2^{1/\gamma^\prime}}\right)
		=-\bigl(\gamma^\prime-\tilde{\gamma}\bigr)
		\omega\left(\frac{t^{1/\gamma^\prime}}{B_2^{1/\gamma^\prime}}\right).
	\end{equation*}
	Since $\gamma^\prime>\widetilde{\gamma}$ and the associated weight function $\omega(s)$ grows faster
	than any $\log s^p$, $p\in\N$, we conclude that the integral in \eqref{FinalEstimate} converges
	and therefore $u\in\vRou{\widetilde{\bM}}{\Omega}$ because $\supp u\subseteq  B_{0}$.
	Since $\widetilde{\bM}\lhd\bM$ we obtain that $u\in\vBeu{\bM}{\Omega}$.
\end{proof}
If we apply Theorem \ref{FiniteVersion} to the Gevrey case we obtain the following corollary.
\begin{Cor}\label{GevreyCor}
	Let $1\leq r<s$ and $\Omega\subseteq\R^n$ be an open set. 
	If $P$ is a non-elliptic differential operator with coefficients in $\G^r(\Omega)$ then
	there is a smooth function $u\in\E(\Omega)$ such that 
	\begin{equation*}
		u\in\G^s(\Omega;P)\!\setminus\!\G^s(\Omega).
	\end{equation*}
\end{Cor}
\begin{proof}
	Since the case $r=1$ is just \cite[Theorem 2.3]{doi:10.1080/03605307808820078}, we can assume
	that $r>1$.
	If we put $\bM=\bG^s$ then $\bG^r=\bM^{\rho}$ with $\rho=r/s$. 
	In particular $1<1/  \rho< s=\gamma(\bM)$. Hence the assertion follows from
	Theorem \ref{FiniteVersion}.
\end{proof}
Combining Corollary \ref{GevreyCor} with \cite[Theorem 1.1]{MR548225} gives immediately
Theorem \ref{GevreyThm}.

\section{Optimal functions in Denjoy-Carleman classes}\label{OptimalSection}
\begin{Def}
	Let $\bN$ be a weight sequence. We say that a smooth function $g\in\E(\Omega)$ on an open
	set $\Omega\subseteq\R^n$ is an optimal function of the Roumieu class $\Rou{\bN}{\Omega}$
	if $g\in\Rou{\bN}{\Omega}$ but $g\notin\Rou{\bT}{\Omega}$ for any
	weight sequence $\bT\precnapprox\bN$.
\end{Def}
 We can use the general approach to the construction of the function $u$ in
 Section \ref{proofs} to define optimal functions
	in Denjoy-Carleman classes.
	Indeed, assume that $\bN$ is a weight sequence which satisfies \eqref{DerivClosed} and
	$\gamma(\bN)>0$. Then there exists an optimal $\{\bN\}$-flat function $G_\bN$ which is defined
	in some sector $S_\gamma$, $\gamma<\gamma(\bN)$. 
	Now let $x_0\in\R^n$ and $\xi_0\in S^{n-1}$.
	If we set
	\begin{equation*}
		g(x)=\int_0^\infty\negthickspace\Phi_\bN(t)e^{it\xi_0(x-x_0)}\,dt, 
	\end{equation*}
where $\Phi_\bN(s)=G_\bN(1/s)$,
then $g$ is a smooth function defined on $\R^n$.
 We see as before that
\begin{equation*}
	D_{\xi_0}^kg(x_0)=\int_0^\infty\negthickspace t^k\Phi_\bN(t)\,dt
\end{equation*}
and thence, according to Proposition \ref{SanzProp},
\begin{equation*}
	D_{\xi_0}^kg(x_0)\geq B_1^{k+1}N_k
\end{equation*}
for some constant $B_1>0$. Hence $g$ cannot be of class $\{\bT\}$ near $x_0$ for any weight sequence
$\bT\precnapprox\bN$.
On the other hand we can show that $g$ is of class $\{\bN\}$ in $\Omega$ if $\bN$ satisfies
\eqref{DerivClosed}. Indeed we have that
\begin{equation*}
	\abs{D^\alpha g(x)}\leq \int_0^\infty\negthickspace t^\alp\Phi_\bN(t)\,dt
	\leq CQ_2^\alp N_\alp.
\end{equation*} 
for some constants $C,Q_2>0$ by Proposition \ref{SanzProp}.
Hence $g$ is an optimal function in the class $\Rou{\bN}{\R^n}$.
\begin{Rem}
	We may point out that this approach defining optimal functions in Denjoy-Carleman classes
	is more flexible than the usual one by Fourier series, see e.g.\ \cite{MR2384272}.
	For example, we can specify not only the point but also the direction where
	the optimality occurs.
	Depending on the problem, one can further modify the construction of optimal functions.
	Indeed, we can show, for example, that the function $u$ from Theorem \ref{MainThm}, if the weight sequence
	$\bM$ satisfies \eqref{DerivClosed}, is an optimal function in a Denjoy-Carleman class
	$\Rou{\bN}{\Omega}$ given by a weight sequence $\bN$ depending on $\bM$ and the operator $P$.
\end{Rem}
In fact, in section \ref{proofs} the function $u$ was constructed in the following way:
For a given weight sequence $\bM$ we choose suitable weight sequences $\bL$ and $\bN$ such
that $\bL\preceq\bM\precnapprox\bN$, $\bL$ is non-quasianalytic and $\gamma(\bN)>0$.
Then we define
\begin{equation*}
	u(x)=\int_1^\infty\negthickspace \psi(t^\eps(x-x_0))\Phi_\bN(t)e^{it\xi(x-x_0)}\,dt
\end{equation*}
where $0<\eps<1$ is some constant, $\xi_0\in S^{n-1}$, $x_0\in\Omega$ and 
$\psi\in\Rou{\bL}{\R^n}$ is such that $\psi(y)=1$ for $\abs{y}\geq \delta$ and $\psi(y)=0$
for $\abs{y}\geq 2\delta$ with $\delta>0$ being such that $B_0=\Set{x\in\R^n\given \abs{x-x_0}\geq 2\delta}\subseteq\Omega$. Furthermore $\Phi_\bN(t)=G_\bN(1/t)$ where $G_\bN$ is an
optimal $\{\bN\}$-flat function.
From the estimate \eqref{OptimalEstimate} it follows that $u\notin\Rou{\bT}{\Omega}$
for any weight sequence $\bT\precnapprox\bN$.
Here we claim that if $\bN$ satisfies \eqref{DerivClosed} then $u\in\Rou{\bN}{\Omega}$.
For this we consider the iterates $D_j^ku$ of $u$ with respect to the constant coefficient operator 
$D_j$, $j\in\Set{1,\dotsc,n}$.
We see that 
\begin{equation*}
	D_j^ku(x)=\int_1^\infty\negthickspace \Delta_k^j(x,t)\Phi_\bN(t)e^{it\xi_0(x-x_0)}\,dt
\end{equation*}
where the functions $\Delta_k^j$ are iteratively given by
\begin{align*}
	\Delta_0^j(x,t)&=\psi(t^\eps(x-x_0)), & j&=1,\dotsc,n,\\
	\Delta_{k+1}^j(x,t)&=D_j\Delta^j_k(x,t)+t\xi_{0,j}\Delta^j_k(x,t), & j&=1,\dotsc,n,\;k\in\N_0,
\end{align*}
where $\xi_{0,j}$ is the $j$-th component of the vector $\xi_0$.
For these functions we have a statement analogous to Lemma \ref{IteratesLemma}:
\begin{Lem}\label{D-AuxLemmaAlt}
	For each $j\in\Set{1,\dotsc,n}$ there is a constant $A_j>0$ such that 
		\begin{equation}\label{Important}
		\abs*{D^\nu_x\Delta_k^j(x,t)}\leq C_0(h_0t^\eps)^{\nut}A_j^k
		\left(t^kL_{\nut}+t^{\eps k}L_{\nut+k}\right),\qquad x\in B_0,\;t\geq 1,
	\end{equation}
	for all $\nu\in\N_0^n$ and $k\in\N_0$. Here $C_0,h_0$ are the constants from \eqref{PsiEstimate}.
\end{Lem}
\begin{proof}
	We will prove the Lemma by induction in $k\in\N_0$.
	In the case $k=0$ we have $\Delta_0^j(x,t)=\psi(t^\eps(x-x_0))$.
	Thence 
	\begin{equation*}
		D_x^\nu\left(\psi(t^\eps(x-x_0)\right)=t^{\eps\nut}D^\nu\psi(t^\eps(x-x_0))
	\end{equation*}
and thus \eqref{PsiEstimate} gives that
\begin{equation*}
	\abs*{D^\nu_x\Delta_0^j(x,t)}\leq C_0(h_0t^\eps)^{\nut}L_{\nut}.
\end{equation*}
We will now show that if \eqref{Important} holds for $k$ then \eqref{Important} is also 
satisfied for $k+1$ when we choose $A_j$ suitably. For simplicity we set $R=t^\eps$.
From Lemma \ref{D-AuxLemma1} we obtain that
\begin{equation*}
	tL_{\nut+k}R^k\leq t^{k+1}L_{\nut}+L_{\nut+k+1}R^{k+1}.
\end{equation*}
We set
\begin{equation*}
	\Theta(k,\nu)=t^kL_{\nut}+R^kL_{\nut+k}
\end{equation*}
and conclude that 
\begin{equation}\label{Constant1}
	t\Theta(k,\nu)\leq t^{k+1}L_{\nut}+tL_{\nut+k}R^k\leq 2\Theta(k+1,\nu).
\end{equation}
For all $\alp\leq 1$ Lemma \ref{D-AuxLemma1} gives the following estimate:
\begin{equation}\label{Constant2}
	\begin{split}
		t^{1-\alp}R^\alp\Theta(k,\nu+\alpha)
		&\leq t^{k+1-\alp}L_{\nut+\alp}R^\alp
		+t^{1-\alp}L_{\nut+\alp+k}R^{\alp+k}\\
		&\leq 2\left(t^{k+1}L_{\nut}+L_{\nut+k+1}R^{k+1}\right)\\
		&=2\Theta(k+1,\nu)
	\end{split}
\end{equation}
Furthermore if $e_j$ denotes the $j$-th unit vector in $\R^n$ then we have
\begin{equation*}
	D_x^{\nu}\Delta_{k+1}^j(x,t)=D_x^{\nu+e_j}\Delta_k^j(x,t)+t\xi_{0,j}D_x^\nu\Delta_k^j(x,t).
\end{equation*}
Thus the induction hypothesis implies that
\begin{equation*}
	\begin{split}
	\abs*{D_x^\nu\Delta_{k+1}^j(x,t)}&\leq \abs*{D_x^{\nu+e_j}\Delta^j_k(x,t)}+\abs{t\xi_{0,j}D_x^\nu
		\Delta^j_k(x,t)}\\
	&\leq C_0(h_0t^\eps)^{\nut+1}A_j^k\left(t^kL_{\nut+1}+t^{\eps k}L_{\nut+k+1}\right)
+C_0(h_0t^\eps)^{\nut}A_j^kt
\abs{\xi_{0,j}}\left(t^{k}L_{\nut}+t^{\eps k}L_{\nut+k}\right)\\
&=C_0(h_0t^\eps)^{\nut}A_j^k\left(h_0t^\eps\Theta(k,\nu+e_j)+\abs{\xi_{0,j}}t\Theta(k,\nu)\right).
\end{split}
\end{equation*}
Now, if we recall that $R=t^\eps$ and $0\leq\abs{\xi_{0,j}}\leq 1$ then using
\eqref{Constant1} and \eqref{Constant2} allows us to conclude that
\begin{equation*}
	\abs*{D_x^\nu\Delta_{k+1}^j(x,t)}\leq C_0(h_0t^\eps)^{\nut}A_j^k2\left(h_0+
	\abs{\xi_{0,j}}\right)
	\Theta(k+1,\nu).
\end{equation*}
Therefore we have proven the Lemma if we choose $A_j\geq 2(h_0+\abs{\xi_{0,j}})$.
\end{proof}
If we set $\nu=0$ in Lemma \ref{D-AuxLemmaAlt} then 
\begin{equation*}
	\abs*{\Delta_k^j(x,t)}\leq C_0A_j^k\left(t^k+t^{\eps k}L_k\right).
\end{equation*}
Thus we have  that
\begin{equation*}
	\abs*{D_j^k u(x,t)}\leq C_0A_j^k\left(\int_1^\infty\negthickspace t^k\Phi_\bN(t)\,dt
	+L_k\int_1^\infty\negthickspace t^{\eps k}\Phi_\bN(t)\,dt\right)
\end{equation*}
for all $k\in\N_0$ and $j=1,\dotsc,n$. If $\bN$ satisfies \eqref{DerivClosed} then
by Proposition \ref{SanzProp} there are
constants $C_1,B_1$ such that 
\begin{equation*}
	\int_1^\infty\negthickspace t^k\Phi_\bN(t)\,dt\leq \int_0^\infty\negthickspace
	t^k\Phi_\bN(t)\,dt\leq C_1B_1^k N_k.
\end{equation*}
If there were constants $C_2,B_2>0$ such that 
\begin{equation}\label{LastEstimate}
	L_k\int_1^\infty\negthickspace t^{\eps k}\Phi_\bN(t)\,dt\leq C_2B_2^k N_k
\end{equation}
then we would be able to show that for each $j=1,\dotsc,n$ there are constants $C_3,h_j>0$ such that
\begin{equation*}
	\abs*{D_j^k u(x)}\leq C_3 h_j^k N_k,
\end{equation*}
i.e.~$u\in\vRou[D_j]{\bN}{\Omega}$ for all $j=1,\dotsc,n$. 

By our assumptions we know that $\bN$ is non-quasianalytic since $\bL$ is non-quasianalytic.
This further implies that $\bN$ satisfies \eqref{AnalInclusion}.
Because we suppose here also that \eqref{DerivClosed} holds for $\bN$ we know 
by \cite[Corollaire 3]{MR557524} (cf.~also \cite[Corollary 3.17]{Fuerdoes2022})  that
\begin{equation*}
	\bigcap_{j=1}^n \vRou[D_j]{\bN}{\Omega}=\Rou{\bN}{\Omega}.
\end{equation*}
Thence $u\in\Rou{\bN}{\Omega}$.
\begin{Thm}\label{OptimalThm1}
	Assume that the hypothesis of Theorem \ref{TechnicalTheorem} holds.
	If the weight sequence $\bN$ satisfies additionally \eqref{DerivClosed} then
	the function $u$ from Theorem \ref{TechnicalTheorem} is an optimal function of 
	$\Rou{\bN}{\Omega}$.
\end{Thm}
\begin{proof}
By assumption we have that there are weight sequences $\bL$, $\bM$, $\bN$, $\bV$
and constants $d\in\N_0$, $A>1$ and $1<\tau< 2d/(2d-1)$ such that
$\gamma(\bN)>0$, $\bM\preceq\bN$, $\bL$ is non-quasianalytic, $\bV\leq \bM$ and $\bL\bV\preceq\bM$.
Furthermore $\bN\leq A\bV^\tau$ and in the proof of Theorem \ref{TechnicalTheorem} we have chosen
$\eps$ such that $\tau<d/(d-\eps)< 2d/(2d-1)$.

Since we suppose also that \eqref{DerivClosed} holds for $\bN$ we have only to show \eqref{LastEstimate} by the deliberations above.
By \eqref{Kernelestimates2}, the definition of $\omega_\bV$ and the fact that $\bL\bV\preceq\bM\precnapprox\bN$ there are constants $C_2,B_2,B>0$ such that
\begin{equation*}
	L_k\int_1^\infty\negthickspace t^{\eps k}\Phi_\bN(t)\,dt\leq C_2B_2^k
	N_k\int_{1}^{\infty}\negthickspace \exp\left(-\omega_\bN\left(\frac{t}{B}\right)+\omega_\bV(t^\eps)\right)\,dt.
\end{equation*}
Now, since $\bN\leq A\bV^{\tau}$ and $\tau<d/(d-\eps)<1/\eps$ for all $d\geq 1$, (by our choice above we have that $\eps<1/2$) we can apply Lemma \ref{auxlemma} and obtain that there is
some constant $C>0$ such that
\begin{equation*}
	\omega_\bV(s)\leq \tau^{-1}\omega_\bN\left(\frac{s^{1/\eps}}{B}\right)+C
\end{equation*}
for all $s\geq 0$. Thence 
\begin{equation*}
	\int_1^\infty\negthickspace\exp\left(-\omega_\bN\left(\frac{t}{B}\right)
	+\omega_\bV(t^\eps)\right)\,dt
	\leq C\int_1^\infty\negthickspace \exp\left(-(1-\tau^{-1})\omega_\bN\left(\frac{t}{B}\right)\right)\,dt
\end{equation*}
and the integral on the right-hand side converges since $\tau>1$ and $\omega_\bN(s)$
increases faster than any power of $\log s$.
Thus we have proven \eqref{LastEstimate} in this case.
\end{proof}
\begin{Cor}
	Suppose that the hypothesis of Corollary \ref{GammaInfinity} holds.
	If the weight sequence additionally satisfies \eqref{DerivClosed} then
	the function $u$ in the statement of Corollary \ref{GammaInfinity} is
	an optimal function in $\Rou{\bM^\nu}{\Omega}$ for some $\nu>1$.
\end{Cor}
\begin{proof}
	Observe that if $\bM$ satisfies \eqref{DerivClosed} then \eqref{DerivClosed} holds for 
	any weight sequence $\bM^r$, $r>0$.
	The statement then follows from the proof of Corollary \ref{GammaInfinity} if we also 
	take Theorem \ref{OptimalThm1} into account.
\end{proof}
\begin{Thm}
	Suppose that the assumptions of Theorem \ref{FiniteVersion} hold.
	If the weight sequence $\bM$ satisfies also \eqref{DerivClosed} then 
	there is some $\nu>1$ such that the function $u$ in Theorem \ref{FiniteVersion}
	is an optimal function of the class $\Rou{\bM^\nu}{\Omega}$.
\end{Thm}
\begin{proof}
	By assumption $1<\gamma=\gamma(\bM)<\infty$ and as in the proof of 
	Theorem \ref{FiniteVersion} we set $\bT=\bM^{1/\gamma}$.
	We use the same choices of $\gamma_0$, $\widetilde{\gamma_0}$, $\gamma^\prime$ and $\eps$
	as in the proof of Theorem \ref{FiniteVersion} and 
	set $\bL=\bT^{\gamma_0}$, $\bV=\bT^{\tilde{\gamma}-\gamma_0}$ and $\bN=\bT^{\gamma^\prime}$.
	Note that $\bN=\bM^{\nu}$ with $\nu=\gamma^\prime/\gamma>1$ since $\gamma^\prime>\gamma$
	and therefore $\bN$ satisfies \eqref{DerivClosed} if $\bM$ does.
	
	With these choices of $\bL$, $\bV$, $\bN$ and $\eps$ we have defined $u$ in the proof
	of Theorem \ref{FiniteVersion} by \eqref{FctDef}. 
	Thus using the arguments before Theorem \ref{OptimalThm1} we have only to check if
	\eqref{LastEstimate} holds in order to show that $u$ is an optimal function 
	of the class $\Rou{\bN}{\Omega}$.

By \eqref{Kernelestimates2} we have that there are constants $A,B>0$ such that
\begin{equation*}
	\abs*{\Phi_\bN(t)}\leq A\exp\left(-\omega_\bN\left(\frac{t}{B}\right)\right)
\end{equation*}
for all $t\geq 0$. Now we use
the fact that
\begin{equation*}
	\frac{t^{\eps k}}{B^{\eps k}}\leq V_k\exp\left(\omega_\bV\left(\frac{t^\eps}{B^\eps}\right)\right)
\end{equation*}
to obtain that 
\begin{equation}\label{LastEstimate2}
	L_k\int_1^\infty\negthickspace t^{\eps k}\Phi_\bN(t)\,dt\leq C_2B_2^k
	N_k\int_{1}^{\infty}\negthickspace \exp\left(-\omega_\bN\left(\frac{t}{B}\right)+
	\omega_\bV\left(\frac{t^\eps}{B^\eps}\right)\right)\,dt
\end{equation}
for some constants $C_2,B_2>0$ since $\bV\leq\bM$ and $\bL\bV\preceq\bM$.
If $\omega$ is the weight function associated to $\bT$ then we recall for the weight functions associated to the weight sequences $\bV$ and $\bN$ the following identities from 
the proof of Theorem \ref{FiniteVersion}:
\begin{align*}
	\omega_{\bN}(s)&=\gamma^\prime\omega\left(s^{1/\gamma^\prime}\right),\qquad s\geq 0,\\
	\omega_{\bV}(s)&=\gamma_1\omega\left(s^{1/\gamma_1}\right), \qquad s\geq 0,
\end{align*}
where $\gamma_1=\widetilde{\gamma}-\gamma_0$.
 By our choices we have that $\gamma_1/\eps=(2d\widetilde{\gamma}-\gamma_0)/d$ 
 and thus $\gamma_1/\eps \geq \gamma^\prime$. Therefore we conclude that
\begin{equation*}
	\omega_\bV\left(\frac{t^\eps}{B^\eps}\right)
	=\gamma_1\omega\left(\frac{t^{\eps/\gamma_1}}{B^{\eps/\gamma_1}}\right)
	\leq \gamma_1\omega\left(\frac{t^{1/\gamma^\prime}}{B^{1/\gamma^\prime}}\right) 
\end{equation*}
for $t\geq B$
since $\omega$ is an increasing function on $(0,\infty)$.
Hence
\begin{equation*}
	-\omega_{\bN}\left(\frac{t}{B}\right)+\omega_\bV\left(\frac{t^\eps}{B^\eps}\right)
	\leq -(\gamma^\prime-\gamma_1)\omega\left(\frac{t^{1/\gamma^\prime}}{B^{1/\gamma^\prime}}\right)
\end{equation*}
for $t\geq B$
and since $\gamma^\prime>\gamma_1$  it follows that
the integral on the right-hand side of \eqref{LastEstimate2} converges.
Thus we have proven \eqref{LastEstimate}.
\end{proof}

\bibliographystyle{abbrv}
\bibliography{vectors}
\end{document}